\documentclass[11pt,a4paper,reqno]{amsart}
\usepackage{amsmath}
\usepackage{amsfonts}
\usepackage{amssymb}
\usepackage{amscd}
\usepackage{url}
\usepackage{enumerate}
\usepackage[pdftex,bookmarks=true]{hyperref}

\newcommand\R{{\mathbf{R}}}

\renewcommand\P{{\mathbf{P}}}
\newcommand\E{{\mathbf{E}}}

\newcommand\diag{{\operatorname{diag}}}

\newcommand\col{{\mathbf{c}}}
\newcommand\row{{\mathbf{r}}}
\newcommand\perm{{\mathbf{perm}}}
\newcommand\per{{\mathbf{pm}}}
\newcommand\sm{{\mathbf{sm}}}
\newcommand\gm{{\mathbf{gm}}}
\newcommand\dt{{\mathbf{det}}}

\newcommand\eps{\varepsilon}

\newcommand\Bd{{\mathbf d}}
\newcommand\Be{{\mathbf e}}

\newcommand\Bx{{\mathbf x}}
\newcommand\By{{\mathbf y}}

\newcommand\BN{{\mathbf N}}

\newcommand\CA{{\mathcal A}}
\newcommand\CB{{\mathcal B}}

\newcommand\CE{{\mathcal E}}

\newcommand\CG{{\mathcal G}}

\newcommand\CP{{\mathcal P}}

\newcommand\CS{{\mathcal S}}

\newcommand\CX{{\mathcal X}}
\newcommand\CY{{\mathcal Y}}

\newcommand\Ext{\mathcal Ext}

\theoremstyle{plain}
  \newtheorem{theorem}[subsection]{Theorem}
  \newtheorem{conjecture}[subsection]{Conjecture}

  \newtheorem{lemma}[subsection]{Lemma}
  \newtheorem{corollary}[subsection]{Corollary}

  \newtheorem{claim}[subsection]{Claim}

\theoremstyle{definition}

\begin{document}

\title[A general law of large permanent]{A general law of large permanent}

\author{J\'ozsef Balogh}
\address{Department of Mathematical Sciences, University of Illinois at Urbana-Champaign, Urbana, Illinois 61801}
\email{jobal@math.uiuc.edu}

\author{Hoi Nguyen}
\address{Department of Mathematics, The Ohio State University, Columbus, Ohio 43210}
\email{nguyen.1261@math.osu.edu}


\thanks{J.~Balogh is partially supported by NSF
 grant DMS-1500121,  Arnold O. Beckman Research Award (UIUC Campus Research Board 15006) and Langan Professional Scholarship.
  H.~Nguyen is partially supported by
  NSF grant DMS-1600782.}

 \maketitle

\begin{abstract}
In this short note we establish a law of large permanent for matrices with entries from an $\BN^2$-indexed stochastic process. This answers a question by Bochi, Iommi and Ponce in \cite{BIP}.
 \end{abstract}

\section{Introduction}\label{section:intro}

Let $M_n=(m_{ij})_{1\le i,j\le n}$ be a square matrix of size $n$ of
real-valued entries. The
permanent of $M_n$ is defined as 

$$\perm(M_n) := \sum_{\pi \in S_n} m_{1\pi(1)} \cdot \ldots\cdot  m_{n \pi(n)}.$$

Let $\Omega_n$ denote the set of  doubly stochastic matrices
$M_n=(m_{ij})_{1\le i,j\le n}$ of size
$n$, that is $0\le m_{ij} \le1$ and $\sum_{j=1}^n m_{ij} =1 $ for every $1\le i\le n$ and $\sum_{i=1}^n
m_{ij}=1$ for every $ 1\le j\le n$. It is well-known that if $M_n\in \Omega_n$ then

\begin{equation}\label{eqn:doubly:W}
\frac{n!}{n^n} \le \perm(M_n) \le 1.
\end{equation}

The upper bound of \eqref{eqn:doubly:W} is elementary, which is attained at permutation matrices; one can also obtain a stability-type
result in this direction: for instance \cite[Theorem A.1]{AN} shows that if $\perm(M_n) \ge n^{-O(1)}$, then all but at most $O(\log n)$ rows (and columns) of $M_n$ contain an entry that is at least 0.9.  

The lower bound of \eqref{eqn:doubly:W} was conjectured by van der Waerden in 1926 \cite{W} and proved in 1981 independently by
Egorychev \cite{E}
and Falikman \cite{F}. Moreover, the minimum of the permanent on $\Omega_n$ is
attained at the matrix $J_n$ of entries $1/n$. 

Following \cite{BIP}, we denote the {\it permanental mean} of a
matrix $M_n$ by 

$$\per(M_n) := \left(\frac{\perm(M_n)}{n!}\right)^{1/n}.$$

Thus for doubly stochastic matrices $M_n\in \Omega_n$, by \eqref{eqn:doubly:W}

\begin{equation}\label{eqn:llp:perm}
1\le \lim_{n\to \infty}  \per(n M_n) \le e.
\end{equation}

Note that by definition, if $M_n=D_n G_n E_n$ where $D_n=\diag(d_i)_{1\le i\le n}, E_n=\diag(e_i)_{1\le i\le n}$ with $d_i,e_i>0$ (in other words, $m_{ij} = d_i g_{ij} e_j$), then 

\begin{equation}\label{eqn:llp:perm:scaling}
\per(M_n)=\gm(\Bd) \per(G_n) \gm(\Be),
\end{equation}

where $\Bd=(d_1,\dots,d_n)$ and $\Be=(e_1,\dots,e_n)$, and $\gm(.)$ is the geometric mean

$$\gm(\Bd):=(\prod_i d_i)^{1/n}, \mbox{ and } \gm(\Be):=(\prod_i e_i)^{1/n}.$$

In other words, \eqref{eqn:llp:perm:scaling} says that the permanental mean $\per(.)$ is homogeneous under matrix scaling.

In this note we will try to relate the permanental mean for matrices of non-negative entries to the so-called {\it scaling mean} $\sm(M_n)$, which is in turn defined as 

$$ \sm(M_n):= \frac{1}{n^2} \inf_{\Bx,\By \in \R_{>0}^n} \frac{\Bx^T M_n
  \By}{\gm(\Bx) \gm(\By)},$$

where $\Bx=(x_1,\dots, x_n)$ and $\By=(y_1,\dots,y_n)$.

Thus the scaling mean can be obtained via an optimization problem. An extremely nice property of the scaling mean for doubly stochastic matrices $M_n$ is that (by using AM-GM inequality)

\begin{equation}\label{eqn:llp:sm}
\sm(n M_n) =1. 
\end{equation}

Another important property, which is again not hard to show, is that $\sm(.)$ is also homogeneous under matrix scaling. In other words, if $M_n=D_n G_n E_n$, with $D_n=\diag(d_i)_{1\le i\le n}, E_n=\diag(e_i)_{1\le i\le n}$ and $d_i,e_i>0$, then 

\begin{equation}\label{eqn:DGE}
\sm(M_n)= \gm(\Bd) \sm(G_n) \gm(\Be).
\end{equation}

The forms $M_n=D_nG_nE_n$ in \eqref{eqn:llp:perm:scaling} and \eqref{eqn:DGE}, with doubly stochastic $G_n$, is called {\it Sinkhorn decomposition}. Not every matrix $M_n$ of non-negative entries can be scaled back to a doubly stochastic matrix through  Sinkhorn decompositions. However, the following beautiful theorem provides a necessary and sufficient condition.

\begin{theorem}\cite{BPS,MO,SK}\label{theorem:Sinkhorn:matrix} A matrix
  $M_n=(a_{ij})_{1\le i,j\le n}$ with non-negative entries has a
  Sinkhorn decomposition $M_n=D_nG_nE_n$ if and only if for each
  positive element $m_{ij}>0$ there exists a permutation $\pi \in S_n$
  such that $\pi(i)=j$ and $m_{1\pi(1)},\dots, m_{n \pi(n)}$ are all
  positive. Moreover, the doubly stochastic matrix $G_n$ is unique and
  the map $M_n \to G_n$ is continuous.
\end{theorem}

We refer the reader to \cite{BIP,FK,FLT,Men,S1} for the history and further developments of Theorem~\ref{theorem:Sinkhorn:matrix}, and also to \cite{LSW,Wig} for algorithmic aspects of this fundamental decomposition.

We put here together the relations between \eqref{eqn:llp:perm},\eqref{eqn:llp:perm:scaling}, \eqref{eqn:llp:sm} and \eqref{eqn:DGE}.

\begin{claim}\label{claim:llp:matrix} Assume that $\{M_n\}_{1\le n< \infty}$ is a sequence of matrices of non-negative entries with a Sinkhorn decomposition $M_n=D_n G_n E_n$. Then

$$1\le \lim_{n\to \infty} \frac{\per(M_n)}{\sm(M_n)}\le e.$$
\end{claim}

\subsection{Main result} One of the main goals of our note, and also of the mentioned paper  \cite{BIP} by Bochi, Iommi and Ponce, is to show that the limit in Claim \ref{claim:llp:matrix} is exactly one in the general context of ergodic theory. 

To prepare for the main statement, we still need to introduce the scaling limit for functions.  Fix a probability
space $(\Omega, \CA, \P)$. Let $\CG(\P)$ denote the set of positive
measurable functions $\varphi: \Omega \to \R_{>0}$ such that $\log \varphi
\in L_1(\P)$. The {\it geometric mean} of $\varphi$ is defined as 

$$\gm(\varphi) := \exp(\int \log \varphi d \P).$$

Let us also fix a pair of $\sigma$-algebras $\CA_1,\CA_2 \subset
\CA$. For $i=1,2$, we define 

$$\CG_i:=\Big\{ \varphi: \Omega \to \R_{>0}: \varphi
\mbox{ is $\CA_i$-measurable and $\log \varphi \in L^1(\P)$} \Big\}.$$ 

The {\it scaling limit} of a non-negative measurable function $f: \Omega \to \R$ with respect to $\CA_1,\CA_2$ is then defined as 

$$\sm_{\CA_1,\CA_2}(f):= \inf_{g_i\in \CG_i} \frac{1}{\gm(g_1)
  \gm(g_2)}\int g_1 f g_2 d\P.$$

Now we are ready to introduce the beautiful result by Bochi, Iommi and Ponce from \cite{BIP}. 

Assume that $(X,\CX, \mu)$ and $(Y,\CY,\nu)$ are Lebesgue probability
spaces, and $S: X \to X,\  \ T: Y\to Y$ are measure preserving
transformations. Given a function $f:X \times Y \to \R^+$, for each
$(x,y)\in X\times Y$ and for each integer $n$ we define the matrix $\Box_n f(x,y)$ to be

$$\Box_n f(x,y):=\big(f(S^ix,T^jy)\big)_{0\le i,j \le n-1}.$$

Let $\CA_1$ and $\CA_2$ be the sub-$\sigma$-algebras
formed by the $S$-invariant and the $T$-invariant sets
respectively. Let $\CB(\mu \times \nu)$ denote the set of positive measurable
functions on $X \times Y$ which are essentially bounded away from zero and
infinity.

\begin{theorem}[Law of large permanent]\cite[Theorem 4.1]{BIP}\label{theorem:llp:BIP} If $S$ and $T$ are ergodic and
  $f\in \CB(\mu \times \nu)$ then for $\mu\times \nu$-almost every $(x,y)\in X \times Y$

$$\lim_{n\to \infty} \frac{\per(\Box_n f(x,y))}{\sm_{\CA_1,\CA_2}(f)}=1.$$
 \end{theorem}

In particular, Theorem \ref{theorem:llp:BIP} not only shows the  $(\mu \times \nu)$-a.~e.~existence of the limit of
the permanental mean of $\Box_n f(x,y)$, but it also indicates that this limit is precisely the scaling mean of $f$,  therefore the result connects the limit to an optimization problem. 

We also invite the reader to \cite[Section 5]{BIP} for applications to Muirhead means as well as to a classical result of Hal\'asz and Sz\'ekely \cite{HSz}.

We now introduce a generalization of Theorem \ref{theorem:llp:BIP}. Suppose that $T$ is an ergodic measure-preserving action on the semigroup $\BN^2$ on a Lebesgue
probability space $(\Omega, \CA, \P)$. Given a function $f:\Omega \to
\R^+$, define the matrix $\Box_nf(\omega)$ to be

$$\Box_nf(\omega): = \big(f(T^{(i,j)}(\omega) )\big)_{0\le
  i,j\le n-1}.$$

 Let $\CA_1$ and $\CA_2$ be the sub-$\sigma$-algebras
formed by the $T^{(1,0)}$-invariant and the $T^{(0,1)}$-invariant sets
respectively. The following was conjectured in \ref{theorem:llp:BIP}.

\begin{conjecture}[Law of large permanent, another version]\label{conj:llp} If $\log f \in L_\infty (\P)$ then for $\P$-almost every $\omega$ 

$$\lim_{n\to \infty} \frac{\per(\Box_nf(\omega))}{
\sm_{\CA_1,\CA_2}(f)}=1.$$
\end{conjecture}

Note that this conjecture would imply Theorem \ref{theorem:llp:BIP} with $S=T^{(1,0)}$ and $T=T^{(0,1)}$. Our result confirms this conjecture.

\begin{theorem}[Main result]\label{theorem:llp}
  Conjecture \ref{conj:llp} holds.
\end{theorem}

The rest of the note is devoted  to prove Theorem \ref{theorem:llp}. Although we will use some important ingredients  from the paper \cite{BIP} by Bochi, Iommi and Ponce, our key approach is  quite different from theirs. Roughly speaking, the proof consists of three steps: (1) reduction to doubly stochastic functions, (2) passing to doubly stochastic matrices of bounded entries, (3) establishing upper bound for permanents of such matrices.

\section{Proof of Theorem \ref{theorem:llp}: passing to doubly stochastic functions}

We  record here other key properties of the scaling mean,
these are functional analogues of the results introduced in Section
\ref{section:intro}.

\begin{theorem}\label{theorem:function} Fix a probability space $(\Omega,\CA,\P)$ and a pair of $\sigma$-algebras $\CA_1,\CA_2 \subset
\CA$. The following holds.

\begin{itemize} 
\item (Homogeneity) If $\varphi \in \CG_1$ and $\psi \in \CG_2$ then

\begin{equation}\label{eqn:homogeneity}
\sm_{\CA_1,\CA_2}(\varphi g \psi) = \gm(\varphi) \sm_{\CA_1,\CA_2}(g) \gm(\psi).
\end{equation}
\vskip .1in
\item(Restriction to $L_\infty$) Let 

$$\CB_i:=\Big\{ h: \Omega \to \R^+: h
\mbox{ is $\CA_i$-measurable and $\log h \in L_\infty(\P)$}
\Big \}.$$ 

Then

\begin{equation}\label{eqn:infty}
\sm_{\CA_1,\CA_2}(f)= \inf_{g_i\in \CB_i} \frac{1}{\gm(g_1)
  \gm(g_2)}\int g_1 f g_2 d\P.
\end{equation}
\vskip .1in
\item (Doubly stochastic) If an integrable non-negative function $g$,
  $g:\Omega \to \R_{\ge 0}$, is doubly stochastic with respect to
  $\CA_1$ and $\CA_2$, that is

$$\E (g|\CA_1) = \E (g|\CA_2) =1,\quad  \P-\mbox{almost everywhere},$$ 
then we have the following analogue of \eqref{eqn:llp:sm}

\begin{equation}\label{eqn:doubly}
\sm_{\CA_1,\CA_2}(g)=1.
\end{equation}
\end{itemize}
\end{theorem}

We refer the reader to Propositions 3.2 and 3.3 of \cite{BIP} for proofs of these results.

Next, we will also need a functional version of Theorem
\ref{theorem:Sinkhorn:matrix} regarding the Sinkhorn decomposition.

\begin{theorem}\cite[Theorem 3.6]{BIP} \label{theorem:Sinkhorn:function}
Every $f: \Omega \to \R^+$ such that $\log f \in L_\infty(\P)$ has a Sinkhorn
decomposition, that is there exist functions $\varphi\in \CB_1, \psi \in \CB_2$
and $g$ doubly stochastic with respect to $\CA_1$ and $\CA_2$ such that for $\P$-almost every $\omega$

\begin{equation}
f(\omega) = \varphi(\omega) g(\omega) \psi(\omega).
\end{equation}
\end{theorem}

By Theorem \ref{theorem:Sinkhorn:function}, for $\P$-almost every $\omega$ we can write $f(\omega)=\varphi(\omega) g(\omega)
\psi(\omega)$ for some $\varphi \in \CB_1, \psi \in \CB_2$, and $g$
doubly stochastic, and so

$$f(T^{(i,j)}(\omega)) =\varphi(T^{(i,j)}(\omega)) g(T^{(i,j)}(\omega))
\psi(T^{(i,j)}(\omega)).$$

Now, as $\varphi$ is $\CA_1$-measurable, for $\P$-almost every $\omega$ the following holds for any fixed $j$ and for  every $1\le i\le n$

$$\varphi(T^{(i,j)}(\omega))= \varphi(T^{(i,0)}(T^{(0,j)}(\omega))) = \varphi(T^{(1,0)}(T^{(0,j)}(\omega))= \varphi(T^{(1,j)}(\omega)).$$

Similarly,  because $\psi$ is $\CA_2$-measurable and $T^{(1,0)}$ and $T^{(0,1)}$ commute, for any fixed $i$
and for every   $1\le j\le n$

$$\psi(T^{(i,j)}(\omega))= \psi(T^{(i,0)}(T^{(0,j)}(\omega)) =\psi(T^{(i,0)}(T^{(0,1)}(\omega)) =\psi(T^{(i,1)}(\omega)).$$

As a consequence, for $\P$-almost every $\omega$

$$\per(\Box_n f(\omega))) = \left (\prod_{i=0}^{n-1}
\varphi(T^{(i,1)}(\omega))\right)^{1/n}  \left(\prod_{i=0}^{n-1} \psi
(T^{(1,i)}(\omega))\right)^{1/n}  \per (\Box_n g(\omega)) .$$

Now by the ergodic theorem

$$ \lim_{n\to \infty} \frac{1}{n}\sum_{0\le i\le n-1} \log \varphi(T^{(i,1)}(\omega))) =\log  \gm(\varphi),$$

and similarly 

$$ \lim_{n\to \infty} \frac{1}{n}\sum_{0\le j\le n-1} \log \psi(T^{(1,j)}(\omega))) =\log  \gm(\psi), \mbox{ for } \P-\mbox{almost every } \omega.$$

Thus by \eqref{eqn:homogeneity} of Theorem \ref{theorem:function}, it
suffices to establish Theorem \ref{theorem:llp} for doubly stochastic
function $g$. In other words, by \eqref{eqn:doubly} of Theorem \ref{theorem:function} we will need to show the following.

\begin{theorem}\label{theorem:llp:doubly} Assume that $\log g \in L_\infty(\P)$ and $g$  is doubly stochastic with respect to
  $\CA_1$ and $\CA_2$. Then 
\begin{equation}\label{eqn:llp:doubly}
\lim_{n\to \infty} \per(\Box_ng(\omega))=1 \mbox{ for } \P-\mbox{almost every } \omega.
\end{equation}
\end{theorem}

From now on we assume that there exists $\lambda>1$ such that for $\P$-almost every $\omega$

$$\lambda^{-1} \le g(\omega) \le \lambda.$$

All of the implied constants below are allowed to depend on $\lambda$. For Theorem \ref{theorem:llp:doubly}, by a limiting argument, it suffices to show the following asymptotic analog.

\begin{theorem}\label{theorem:llp:approx} For any $\eps>0$, there exists $n_0=n_0(\eps,\lambda)$ such
  that for any $n\ge n_0$, there exists a
  measurable set $\CE_n$ of measure at most $\eps$ such that for all
  $\omega \notin \CE_n$ we have 

$$1- C \eps \le \per(\Box_n g(\omega)) \le 1+ C \eps,$$

where $C$ is a constant depending on $\lambda$.
\end{theorem}

\section{Proof of Theorem \ref{theorem:llp}: approximation by doubly stochastic matrices of bounded entries}\label{section:llp}

We next show that most of the row sums and column sums of the matrix $\frac{1}{n}\Box_n g(\omega)$ are asymptotically the same.

\begin{lemma}\label{lemma:approx:doubly} With an exception of at most $\eps n$ rows and columns, the following holds for the rows $i$ and  columns $j$ of  the matrix $\Box_n g(\omega)$

$$(1-\eps)n \le \sum_{k=0}^{n-1} g(T^{(i,k)}(\omega)), \quad   \sum_{k=0}^{n-1} g(T^{(k,j)}(\omega)) \le (1+\eps) n.$$

\end{lemma}

\begin{proof}(of Lemma \ref{lemma:approx:doubly}) Set

$$g_n^{(1)}(\omega):= \frac{1}{n}\sum_{i=0}^{n-1}g(T^{(i,0)}(\omega)) \quad \quad 
\text { and } \quad \quad 
g_n^{(2)}(\omega):= \frac{1}{n}\sum_{i=0}^{n-1}g(T^{(0,i)}(\omega)).$$
 
By Birkhoff's ergodic theorem in $L_1$ (see for instance \cite[Theorem 2.1.5]{Keller}),

$$\lim_{n\to \infty}\frac{1}{n}g_n^{(1)} \to \E(g|\CA_1)=1 \mbox{ in } L_1(\P),$$

and 

$$\lim_{n\to \infty}\frac{1}{n}g_n^{(2)} \to \E(g|\CA_2)=1 \mbox{ in } L_1(\P),$$

where we used the fact that $g$ is doubly stochastic with respect
to $\CA_1$ and $\CA_2$. 

Thus for any $\eps>0$, there exists $n_0=n_0(\eps)$ such
that for $n\ge n_0$ we have 

\begin{equation}\label{eqn:L1:g12}
\int_{\Omega} |g_n^{(1)}(\omega)
-1|d\P(\omega) \le \eps^4 \mbox{ and } \int_{\Omega} |g_n^{(2)}(\omega)
-1|d\P(\omega) \le \eps^4.
\end{equation}

Next, define

$$\CE_n^{(1)}:=\Big \{\omega: \frac{1}{n} \sum_{k=0}^{n-1}
|g_n(T^{(0,k)}(\omega))-1| > \eps^2 \Big \}$$

as well as 

$$\CE_n^{(2)}:=\Big \{\omega: \frac{1}{n} \sum_{k=0}^{n-1}
|g_n(T^{(k,0)}(\omega))-1| > \eps^2 \Big \}.$$

By Markov's bound,

\begin{align*}
\P(\CE_n^{(1)}) &\le \eps^{-2} \int_{\Omega} \frac{1}{n} \sum_{k=0}^{n-1}
|g_n^{(1)}(T^{(0,k)}(\omega))-1| d\P(\omega)\\
&\le \eps^{-2} \frac{1}{n}\sum_{k=0}^{n-1} \int_{\Omega} \
|g_n^{(1)}(T^{(0,k)}(\omega))-1| d\P(\omega)\\
&=  \eps^{-2} \int_{\Omega} \
|g_n^{(1)}(\omega)-1| d\P(\omega) \le \eps^2,
\end{align*}

where we just used the fact that $T^{(0,1)}$ is measure preserving together with the bound \eqref{eqn:L1:g12} on the $L_1$-norm of $g_n^{(1)}-1$.

Similarly, we also have 

$$\P(\CE_n^{(2)})\le \eps^2.$$ 

Let $\omega \in \Omega \backslash (\CE_n^{(1)}\cup \CE_n^{(2)})$. By definition

$$ \frac{1}{n} \sum_{k=0}^{n-1} |g_n^{(1)}(T^{(0,k)}(\omega))-1| < \eps^2.$$

Thus by averaging, for all but at most $\eps n$ indices $k \in
\{0,\dots, n-1\}$, $ |g_n^{(1)}(T^{(0,k)}(\omega))-1| < \eps$. In other words, by the definition of $g_n$

$$\left|\sum_{i=0}^{n-1}g(T^{(i,k)}(\omega))-n\right| \le \eps n.$$

Similarly,  for all but at most $\eps n$ indices $k \in
\{0,\dots, n-1\}$, 

$$\left|\sum_{i=0}^{n-1}g(T^{(k,i)}(\omega))-n\right| \le \eps n.$$

\end{proof}

As we have seen from Lemma \ref{lemma:approx:doubly}, most of the row sums and column sums of the matrix $\Box_n (\omega)$ are asymptotically $(1+o(1))n$. In the next lemma we show that this matrix can be approximated by a genuine doubly stochastic
matrix in $L_1$-norm.

\begin{lemma}\label{lemma:L1:doubly} Let $0<\eps< 1<\lambda$ be given positive
  constants, where $\eps$ is sufficiently small depending on $\lambda$. Suppose that $X_n$ is a matrix with the following
  properties
\begin{itemize}
\item $\lambda^{-1} \le x_{ij} \le \lambda$;
\vskip .1in
\item all but at most $\eps n$ rows and columns of $X_n$ have sum
  belonging to the
  range $[(1-\eps) n, (1+\eps)n]$.
\end{itemize}
Then there exists  $X_n' \in n \cdot \Omega_n$ such that $
(2\lambda)^{-1} \le x_{ij}' \le 2\lambda$ for all $i,j$ and

$$\sum_{1\le i,j \le n}|x_{ij}-x_{ij}'| \le 16 \eps \lambda^2  n^2.$$
\end{lemma}

\begin{proof}(of Lemma \ref{lemma:L1:doubly}) We first completely truncate the same number, assuming the worse case $\lceil \eps n \rceil$, of rows and columns whose sums were not in the range $[(1- \eps)n, (1+\eps)n]$. The obtained square matrix has size  $m=n-\lceil 
  \eps n \rceil$ with row and column sums belonging to $[(1-\eps)n -\lambda \lceil \eps n \rceil), (1+\eps)n] \subset [(1-2\lambda \eps)n ,(1+\eps) n]$. We next multiply each row of the obtained matrix by an appropriate factor from the interval $[1-2\eps, 1+2\lambda \eps]$ to make the row sum exactly $m$. Let the new matrix be $Y_m=(y_{ij})_{1\le i,j\le m}$, whose properties are summarized below when $\eps$ was chosen sufficiently small:

\begin{itemize}
\item every row   $\row_i$ for $ 1\le i\le m$  has sum $s(\row_i)$ exactly $m$;
\vskip .1in
\item every column sum for  $s(\col_j), 1\le j\le m$ belongs to the range $[(1-4\lambda \eps)m, (1+4\lambda
  \eps)m]$, 
\vskip .1in
\item for all $i,j$ we have $(2\lambda)^{-1} \le y_{ij} \le 2\lambda$. 
\end{itemize}

We now approximate $Y_m$ by matrices from $m \cdot \Omega_m$.

\begin{claim}\label{claim:B_m'} There exists a matrix $Y_m' \in m \cdot \Omega_m$ such that 
$$\sum_{1\le i,j \le n}|y_{ij}-y_{ij}'| \le 4 \eps \lambda^2  m^2.$$
\end{claim}

It is clear that by the construction of $Y_m$, after gluing $n I_{n-m}$ to $\frac{n}{m}  Y_m'$ one creates a matrix $X_n' \in n \cdot \Omega_n$ which approximates $X_n$ as desired.

It  remains to prove Claim \ref{claim:B_m'}. We are going to modify the column vectors of $Y_m$ so that they all have sum
$m$. Let $\col_i$ and $\col_j$ be two columns with 

$$s(\col_i)<m <s(\col_j).$$ 

{\bf Case 1.} Assume for now that
$s(\col_i)+s(\col_j)\ge 2m$. We are going to modify the entries of
$\col_i,\col_j$ as follows: for $k=1$ to $m$, consider the pair
$(y_{ki},y_{kj})$. Increase $y_{ki}$ by a largest possible amount
$y_k\ge 0$ (and decrease $y_{kj}$ by  $y_k$ accordingly to preserve the row sum) so that $s(\col_i)$
is still below $m$ and then entries are still within $(2\lambda)^{-1} \le y_{ki}', y_{kj}' \le 2
\lambda$. 

We claim that $s(\col_i)=m$ after modifying all $(y_{ki},y_{kj}), 1\le
k\le m$. Assume otherwise, then the reason we are not able to increase
$s(\col_i)$ furthermore is that for any $1\le k\le m$, either
$y_{ki}=2\lambda$ or $y_{kj}=(2\lambda)^{-1}$. But in either case,
$y_{ki}\ge y_{kj}$, and so $s(\col_i)$ must be at least $s(\col_j)$,
a contradiction to our assumption that $s(\col_i) + s(\col_j)\ge 2m$.

{\bf Case 2.} For the case $s(\col_i)+s(\col_j)<2m$, for $k=1$ to $m$
we decrease $y_{kj}$ by a largest possible amount
$y_k\ge 0$ (and increase $y_{kj}$ by $y_k$ accordingly) so that $s(\col_j)$
is still at least $m$, and $(2\lambda)^{-1} \le y_{ki}', y_{kj}' \le 2
\lambda$. Again it is not hard to check that after modifying all $(y_{ki},y_{kj}), 1\le
k\le m$, we will obtain $s(\col_j)=m$.

Note that because of the nature of our process (as we always either increase or decrease
all entries of one column), the $L_1$-distance of the new matrix
$B_m'$ and the original matrix $Y_m$ is bounded by

$$\sum_{i,j}|y_{ij}(Y_m) - y_{ij}(Y_m')| \le 4\eps \lambda^2 m.$$

Now consider the new matrix, if it still does not belong to $m \cdot \Omega_m$
then choose any column pair $(\col_{i'},\col_{j'})$ with
$s(\col_{i'})<m<s(\col_{j'})$ and continue the modifying process as above.

After at most $m$ such iterations, our matrix must belong to $m\cdot
\Omega_m$ because the number of columns of sum $m$ increase by at least one
after each iteration.
\end{proof}

It is plausible to obtain $X_n'$ by considering the Sinkhorn decomposition of $X_n$ but we have not followed this approach. However, in our simple proof above we did use the trick of scaling the rows appropriately. 

For convenience, we gather here an immediate consequence of Lemma \ref{lemma:approx:doubly} and Lemma \ref{lemma:L1:doubly}.

\begin{corollary}\label{cor:A} There exists a matrix $A_n(\omega) =(a_{ij}(\omega))_{1\le i,j\le n}\in n \cdot \Omega_n$ such that  

$$(2\lambda)^{-1} \le a_{ij}(\omega) \le 2\lambda, \mbox{ and } \frac{1}{n^2}\sum_{1\le i,j\le n} |g(T^{(i,j)}(\omega)) - a_{ij}(\omega)| \le 16 \eps \lambda^2.$$
\end{corollary}

We next apply the following nice result by Bochi, Iommi and Ponce \cite[Lemma 4.4]{BIP}.

\begin{lemma}\label{lemma:L1} If $X_n$ and $Y_n$ are matrices with $\lambda^{-1} \le x_{ij}, y_{ij} \le \lambda$, then

$$\left|\log \frac{\per(X_n)}{\per(Y_n)}\right| \le \frac{\lambda^5}{n^2} \sum_{i,j}|x_{ij}-y_{ij}|.$$ 
\end{lemma}

By Corollary \ref{cor:A} and Lemma \ref{lemma:L1} above, to prove Theorem \ref{theorem:llp:approx} we just need to show that, with $A_n$ as in Corollary \ref{cor:A},

$$\per(A_n) = (1+o(1)).$$

This is the content of another question posed by Bochi, Iommi and Ponce. 

\begin{theorem}\label{theorem:doubly}\cite[Conjecture 6.2]{BIP}
Assume that $\{A_n=(a_{ij})\}_{1\le n<\infty}$ is a sequence of matrices of increasing size
$n$ such that $A_n \in n \cdot \Omega_n$. Assume furthermore that
there exists $\lambda>1$ so that $ \lambda^{-1} \le a_{ij}\le \lambda, 1 \le i,j \le n$. Then

$$\lim_{n\to \infty} \per(A_n)  =1.$$
\end{theorem}

Clearly, by the van der Waerden bound \eqref{eqn:doubly:W}, for Theorem \ref{theorem:doubly} the main task is to bound $\per(A_n)$ from above. In the next section we will give a simple proof of it, and hence concluding the proof of Theorem \ref{theorem:llp}.

\section{Proof of Theorem \ref{theorem:llp}: permanent of doubly stochastic matrices of bounded entries}\label{section:stochastic}

Let $\CS_n$ denote the set of (row) stochastic matrices
$M_n=(m_{ij})_{1\le i,j\le n}$ of size
$n$, that is $0\le m_{ij} \le 1$ and $\sum_{j=1}^n m_{ij} =1 $ for every $1\le i\le n$. The following result  immediately implies Theorem \ref{theorem:doubly}.

\vskip .1in
\begin{theorem}\label{theorem:stochastic}
Suppose that $\{A_n\}_{1\le n<\infty}$ is a sequence of matrices of increasing size
$n$ such that $A_n \in n \cdot \CS_n$. Suppose furthermore that
there exists $\lambda>1$ so that $a_{ij} \le \lambda$ for all
$i,j$. Then
$$\perm(A_n) \le  e^{2\lambda} n^{(\lambda-1)/2} n!.$$
\end{theorem}

It remains to prove Theorem \ref{theorem:stochastic}. Denote  $\Ext_n$  the collection of matrices $A_n=(a_{ij})$ where $A_n \in n \cdot
\CS_n, a_{ij}\le \lambda$ and $\perm(A_n)$ is maximum. Note that $\Ext_n$ is non-empty because of
 the compactness.

\begin{lemma}\label{lemma:shifting} The family $\Ext_n$ contains of  a matrix $A_n=(a_{ij})_{1\le i,j\le
    n}$ with the following properties: for every row $1\le i\le n$, all but at most one
  entry  take values either zero or $\lambda$,
  and the (possible)  remaining entry is strictly between zero and $\lambda$.
\end{lemma}

\begin{proof}(of Lemma \ref{lemma:shifting}) It suffices to work with
  the first row. We show that for any $A_n\in \Ext_n$, there is a way
  to force all but at most one of the entries $a_{11},\dots, a_{1n}$
  of its first row to be either zero or $\lambda$, while keeping
  $\perm(A_n)$ to be optimal. 

In what follows we freeze all $a_{ij}$ with $2\le i\le n, 1\le j\le
n$. For $1\le
  i\le n$, let $M_i$ be to $(n-1)\times (n-1)$ minor obtained from
  $A_n$ by deleting its first row and $i$-th column. By the Laplace expansion

$$\perm(A_n) = \sum_{i=1}^n a_{1i} \perm(M_i).$$

Assume that there are two entries among the $a_{11},\dots,
a_{1n}$ that are not either zero or $\lambda$. Without loss of
generality, assume that these are $a_{11}$ and $a_{12}$. Assume
furthermore that 

$$0\le \perm(M_1) \le \perm(M_2).$$

Given the constraint $a_{11}+a_{12} = n -
\sum_{3\le i\le n} a_{1i}$ and $0\le a_{11},a_{12} \le \lambda$
(with temporarily fixed $a_{1i}, 3\le i\le n$), it is easy to see that
the sum
$a_{11}\perm(M_1)+a_{12} \perm(M_2)$ is bounded from above by 

$$
a_{11} \perm(M_1)+a_{12} \perm(M_2) \le 
\begin{cases}
0 \times \perm(M_1) + (a_{11}+a_{12})  \perm(M_2), & a_{11}+a_{12} \le \lambda \\
(a_{11}+a_{12}-\lambda)  \perm(M_1) + \lambda  \perm(M_2),
& \lambda<a_{11}+a_{12} < 2 \lambda.
\end{cases}
$$

In other words, we do not decrease $\perm(A_n)$ by shifting
$(a_{11},a_{12},a_{13},\dots, a_{1n})$ to either
$(0,a_{11}+a_{12},a_{13},\dots, a_{1n})$ or $(a_{11}+a_{12}-\lambda,
\lambda,a_{13},\dots, a_{1n})$. Remark that in either case we increase the
number of entries taking values zero or $\lambda$  in the first row of $A_n$.
\end{proof}

Let $A_n\in \Ext_n$ be a matrix obtained by applying Lemma
\ref{lemma:shifting}. After replacing in each row the exceptional
entry by $\lambda$ (if needed), we obtain a
matrix $A_n'$ whose each row contains $m=\lceil \frac{n}{\lambda}
\rceil$ entries of value $\lambda$ and  $n-m$ entries of  value
zero. Also  $\perm(A_n) \le \perm(A_n')$. 

After scaling down the entries of $A_n'$ by a factor of $\lambda$ (and so the permanent is scaled
down by a factor of $\lambda^n$), we obtain a
$\{0,1\}$-matrix $B_{n}$ where the number of ones in each
row is exactly $m$. We next apply  Bregman-Minc inequality for
  permanent of $\{0,1\}$ matrices (see for instance \cite{B,M} or \cite[(2.1)]{AF})

\begin{theorem}\label{theorem:BM} Let $M_n=(m_{ij})_{1\le i,j\le n}$ be a $\{0,1\}$-matrix. Denote $r_i= \sum_{j=1}^n m_{ij}, 1\le i\le n$. Then we have
$$\perm(M_n) \le \prod_{i=1}^n (r_i!)^{\frac{1}{r_i}},$$
where equality holds if and only if up to permutation of rows and
columns $M_n$ is a block diagonal matrix where each block is a square
matrix of all ones. In other words, equality holds only when $M_n$ is an adjacency matrix
of disjoint union of bipartite $K_{r_i,r_i}$ graphs.
\end{theorem}

Applying Theorem \ref{theorem:BM} to $B_n$ we obtain
 \begin{equation}\label{eqn:B_n}
\perm(B_{n}) \le (m!)^{n/m},
\end{equation}

with equality holds only when $B_n$ is an adjacency matrix
of disjoint union of $K_{m,m}$ graphs.

\begin{proof}(of Theorem \ref{theorem:stochastic}) We have
$$\perm(A_n) \le \perm(A_n')= \lambda^n \perm(B_{n}) \le (\lambda)^n  (m!)^{n/m}.$$
Note that by Stirling's approximation  $\sqrt{2 \pi k} \left(\frac{k}{e}\right)^k \le k!
\le e \sqrt{k}   \left(\frac{k}{e}\right)^k$. Thus,
\begin{align*}
\perm(A_n) \le \lambda^n (m!)^{n/m} &\le  \lambda^n \left(e
\sqrt{m}   \left(\frac{m}{e}\right)^m\right)^{n/m} \le (e \sqrt{m})^\lambda
\left(\frac{n+\lambda}{e}\right)^n\\
&\le  (e \sqrt{m})^\lambda e^{\lambda} \left(\frac{n}{e}\right)^n  \le e^{2\lambda} n^{(\lambda-1)/2} n!.
\end{align*}
%=\footnote{Maybe better to write: $10\cdot e^{2\lambda} n^{(\lambda-1)/2} n!$}

\end{proof}

{\it Added to proof.} After the proof of Theorem \ref{theorem:stochastic} was written, we were informed that the result follows from \cite{S}. However, as our proof looks short and direct, we decided to keep it here for completeness.

\section{Further remarks} 
A crucial problem is to calculate the scaling mean $\sm_{\CA_1,\CA_2}(f)$ for various natural candidates of $T$ and $f$. We refer the reader to \cite{BIP} and \cite{BIP:matching} for many illuminating examples as well as for a fast and simple iterative process regarding this issue. 

Theorem \ref{theorem:llp:BIP} and Theorem \ref{theorem:llp} can be considered as law of large number. It remains an interesting problem to extend these results to central limit theorem for the logarithmic permanents. We hope to address this challenging issue in the near future. In what follows we gather two small applications of our main result.

\subsection{Determinant of gaussian matrices with different variances}
By taking advantage of the explicit approximation
$n!$ of $\perm(A_n)$ in Theorem \ref{theorem:stochastic}, we deduce that the logarithmic of square determinant of random gaussian
matrices with variance profile $A_n$ is concentrated around $\log
n!$. This is by no means fundamental, but we have not found similar
results in the literature.

Let $A_n=(a_{ij})_{1\le i,j\le n}$ be a deterministic matrix with
 non-zero entries. Assume that $X_n=(x_{ij})_{1\le i,j\le n}$ is a random matrix where
 the entries $x_{ij}$ are of the form $x_{ij}
 =\sqrt{a_{ij}}g_{ij}$, with $g_{ij}$ being independent  and identically distributed standard
 gaussian random variable. Theorem~\ref{theorem:stochastic} then yields the following statement.

\begin{corollary}\label{cor:det}Let $\eps>0$ be a constant, and let $\lambda<n$ be a parameter that might depend on $n$. Let the matrix $A_n=(a_{ij})_{1\le i,j
  \le n}\in n \cdot \Omega_n$  be such that
$\eps \le a_{ij} \le \lambda$ for every $1\le i,j\le n$. Then with high probability $\log \dt(X_n)^2$ is concentrated around $\log n!$. More precisely, there exist constants $c_0=c_0(\eps), c_1=c_1(\eps)$ and $c_2=c_2(\eps)$ such that
$$\P\Big(|\log \dt(X_n)^2 -\log n!| \ge c_0 \sqrt{\lambda n} \log^{c_1}n\Big) \le \exp(-c_2 \log^4n).$$
\end{corollary}

Note that when $a_{ij}=1$ (i.e. Ginibre ensemble), Corollary
\ref{cor:det} (or its stronger form) can be deduced directly from the observation of
Goodman \cite{G} that $\dt(X_n)^2 =\prod_{i=1}^n d_i^2$
(where $d_i$ is the distance from the $i$-th row to the subspace generated by the last $n-i$ rows) and that the $d_i^2$ are independent $\chi^2$
 of parameter $i$ thanks to the invariance property of gaussian vectors. However, when the $a_{ij}$'s are not necessarily the same as in Corollary \ref{cor:det}, this invariance property totally breaks down. 

On the other hand, one can still establish concentration for $\log \dt(X_n)^2$ by spectral mean, namely by using the result of Guionnet and Zeitouni \cite{GZ} on the concentration of linear statistics $\sum_i f(s_i)$, where $s_1\ge \ldots \ge s_n$ are the singular values of $X_n$. Although in our case $f(x)=\log x$ blows up at zero and infinity, one can still remove the singularity by truncation and by showing that the last few singular values are bounded away from zero with high probability (treatment for the soft edge is more standard). Such an approach can be found for instance in the work of Friedland, Rider and Zeitouni \cite{FRZ}; we also refer the reader to \cite{Ba,CV,RZ} and the references therein. By this concentration phenomenon, and by the fact that $\E  \dt(X_n)^2 = \perm(A_n)$, one can establish (for a wide range of $A_n$) that with high probability $\log \dt(X_n)^2$ is concentrated around $\log \perm(A_n)$. This was indeed the motivation for the Barvinok-Godsil-Gutman estimator \cite{Ba,GG}.  In this spirit, allow us to cite here a special version of \cite[Corollary 1.5]{RZ} by Ruldelson and Zeitounni, a result directly relevant to our simple goal above.

\begin{lemma} With $A_n$ as in Corollary \ref{cor:det},
$$\P\Big(|\log \dt(X_n)^2 - \log \perm (A_n)| \ge c_0 \sqrt{\lambda n} \log^{c_1}n\Big) \le \exp(-c_2 \log^4n).$$ 
\end{lemma}
Corollary \ref{cor:det} then follows from this result and Theorem~\ref{theorem:stochastic}.

\subsection{Perfect matchings in random bipartite graphs in random environment} This subsection is motivated by another paper of Bochi, Iommi and Ponce \cite{BIP:matching}.

Let  $(\Omega, \CA,\P)$ be a Lebesgue probability space, and let $f: \Omega \to (0,1]$ be a function with $\log f \in L_\infty(\P)$. Suppose that $T$ is an ergodic measure-preserving  action on the semi-group $\BN^2$ on $(\Omega, \CA,\P)$ as in Theorem \ref{theorem:llp}, with $\CA_1$ and $\CA_2$ being the sub-$\sigma$-algebras formed by the $T^{(1,0)}$-invariant and the $T^{(0,1)}$-invariant sets respectively.

For each {\it environment} $\omega$ of the space $\Omega$, for each $n\ge 1$ we define a random bipartite graph $G_n(\omega)$ on the vertex sets $W_n=\{w_1,\dots, w_n\}$ and $M=\{m_1,\dots,m_n\}$ according to the following law $\P_n(\omega)$: for $1\le i,j\le n$, each edge $w_i m_j$ is chosen independently at random with probability 

$$a_{ij}(\omega) := f(T^{(i,j)} (\omega)).$$

Note that our random graph is {\it inhomogeneous} as the $a_{ij}$ can be totally different. It is clear that the number $N$ of perfect matchings in this random bipartite graph is 

$$N= \sum_{\pi \in S_n} \mathbf{1}_{w_1 m_{\pi(1)} \mbox{is an edge}} \cdots \mathbf{1}_{w_n m_{\pi(n)} \mbox{is an edge}}.$$
 
Thus the expected number of perfect matchings in $G_n(\omega)$ with respect to the law $\P_n(\omega)$ is

\begin{align*}
N_n(\omega):=\E_{\P_n(\omega)} N &= \E_{\P_n(\omega)}\sum_{\pi \in S_n} \mathbf{1}_{w_1 m_{\pi(1)} \mbox{is an edge}}   \cdots   \mathbf{1}_{w_n m_{\pi(n)} \mbox{is an edge}}\\
&= \sum_{\pi \in S_n} a_{w_1 m_{\pi(1)}} \cdots a_{w_n m_{\pi(n)}}\\
&= \perm(\Box_n(f(\omega))).
\end{align*}
By using Theorem \ref{theorem:llp}, we  obtain the following variant of law of large number for the number of perfect matchings of random bipartite graphs in random environement.

\begin{theorem} For $\CP$-almost every environment $\omega \in \Omega$, 
$$\lim_{n \to \infty} (\frac{N_n}{n!})^{1/n} = \sm_{\CA_1,\CA_2}(f). $$
\end{theorem}

{\bf Acknowledgements.} The authors are thankful to S.~Leibman,
G.~Iommi  and to the anonymous referee for their  very helpful comments.

\end{document}